\documentclass[12pt,reqno]{article}

\usepackage[usenames]{color}
\usepackage[colorlinks=true,
linkcolor=webgreen, filecolor=webbrown,
citecolor=webgreen]{hyperref}

\definecolor{webgreen}{rgb}{0,.5,0}
\definecolor{webbrown}{rgb}{.6,0,0}

\usepackage{amssymb}
\usepackage{graphicx}
\usepackage{amscd}
\usepackage{lscape}
\usepackage{tikz}
\usepackage{tikz-cd}
\usepackage{pgfplots}

\usetikzlibrary{matrix}
\usetikzlibrary{fit,shapes}
\usetikzlibrary{positioning, calc}
\tikzset{circle node/.style = {circle,inner sep=1pt,draw, fill=white},
        X node/.style = {fill=white, inner sep=1pt},
        dot node/.style = {circle, draw, inner sep=5pt}
        }
\usepackage{tkz-fct}

\usepackage{amsthm}
\newtheorem{theorem}{Theorem}

\newtheorem{proposition}[theorem]{Proposition}
\newtheorem{corollary}[theorem]{Corollary}

\theoremstyle{definition}

\newtheorem{example}[theorem]{Example}

\usepackage{float}

\usepackage{graphics,amsmath}
\usepackage{amsfonts}
\usepackage{latexsym}
\usepackage{epsf}

\DeclareMathOperator{\Rev}{Rev}

\begin{document}

\begin{center}
\vskip 1cm{\LARGE\bf On the Central Description of the Group of Riordan Arrays} \vskip 1cm \large
Paul Barry\\
School of Science\\
Waterford Institute of Technology\\
Ireland\\
\href{mailto:pbarry@wit.ie}{\tt pbarry@wit.ie}
\end{center}
\vskip .2 in

\begin{abstract} We provide an alternative description of the group of Riordan arrays, by using two power series of the form $\sum_{n=0}^{\infty} g_n x^n$, where $g_0 \ne 0$ to build a typical element of the constructed group. We relate these elements to Riordan arrays in the usual description, showing that each newly constructed element is the vertical half of a ``usual'' element. The product rules and the construction of the inverse are given in this new description, which we call a ``central'' description, because of links to the central coefficients of Riordan arrays. This is done for the case of ordinary generating functions. Finally, we briefly look at the exponential case.
\end{abstract}

\section{Preliminaries}
We let $\mathcal{F}=\{ a_0+a_1 x+ a_2 x^2+ \cdots \,| a_i \in \mathbf{R}\}$ be the set of formal power series with coefficients $a_i$ drawn from the ring $\mathbf{R}$. This ring can be any ring over which the operations we will carry out make sense, but for concreteness it can be assumed to be $\mathbb{Q}$. It will be seen that most of the matrices we deal with have integer entries. We shall use two distinguished subsets of $\mathcal{F}$, namely
$$\mathcal{F}_0=\{ a_0+a_1 x+ a_2 x^2+ \cdots \,| a_i \in \mathbf{R}, a_0 \ne 0\},$$ and
$$\mathcal{F}_1=\{ a_1 x+ a_2 x^2+ \cdots \,| a_i \in \mathbf{R}, a_1 \ne 0\}.$$
Throughout our exposition, we will stipulate that $g(x) \in \mathcal{F}_0$, $f(x) \in \mathcal{F}_0$, $u(x) \in \mathcal{F}_0$ and $v(x) \in \mathcal{F}_1$.

A Riordan array \cite{book, SGWW} may be defined by a pair $(u(x), v(x))$ of power series, represented by the invertible lower-triangular matrix $\left(t_{n,k}\right)$ where
$$t_{n,k}=[x^n] u(x)v(x)^k.$$
Here, the functional $[x^n]$ acts on elements of $\mathcal{F}$ by returning the coefficient of $x^n$ of the power series in question \cite{MC}.

A well-known example of a Riordan array is Pascal's triangle $\left(\binom{n}{k}\right)$, which is defined by
$\left(\frac{1}{1-x}, \frac{x}{1-x}\right)$. When Pascal's triangle is represented as a pyramid, the central elements of the pyramid are $\binom{2n}{n}$. We have
$$\binom{2n}{n}=[x^n](1+2x+x^2)^n=[x^n](1+x)^{2n}.$$
Another example of central numbers is that of the trinomial numbers which can be defined by $[x^n](1+x+x^2)^n$.

The set of Riordan arrays defines a group, with the product given by
$$(d(x), h(x)) \cdot (u(x), v(x))=(d(x) u(h(x)), v(h(x))),$$ and inverse given by
$$(u, v)^{-1}=\left(\frac{1}{u(\bar{v})}, \bar{v}\right),$$ where
$\bar{v}=\Rev(v)$ is the compositional inverse of $v$. The identity element of the group is given by $(1,x)$. Multiplication in the group translates into ordinary matrix multiplication of the associated matrix arrays.

An important subgroup of the Riordan group is the hitting-time subgroup \cite{Hitting}, given by the elements $\{\left(\frac{xv'(x)}{v(x)}, v(x) \right) \,|\, v \in \mathcal{F}_1\}$.

For a given Riordan array $\left(t_{n,k}\right)$, the two triangles $H=\left(t_{2n,n+k}\right)$ and $V=\left(t_{2n-k,n}\right)$ are known as the horizontal half and the vertical half of the array. We have the following description of these ``half'' arrays \cite{halves, Half}.
\begin{proposition} The vertical and horizontal halves of the Riordan array $(u, v)$ are given respectively by
$$V=(g(\phi(x)),x)\cdot \left(\frac{x \phi'(x)}{\phi(x)}, \phi(x)\right)=\left(\frac{g(\phi(x))x \phi'(x)}{\phi(x)}, \phi(x)\right),$$ and
$$H=(g(\phi(x)),x) \cdot \left(\frac{x \phi'(x)}{\phi(x)}, v(\phi(x))\right),$$
where
$$\phi(x)=\Rev\left(\frac{x^2}{v(x)}\right).$$
\end{proposition}
We have used the notation $\Rev(v)$ to denote the compositional inverse of $v \in \mathbf{F}_1$. Other notations commonly used are $v^{<-1>}$ and $\bar{v}$.

The goal of this note is to provide a ``central'' description of Riordan arrays.

\section{A motivating example}
Pascal's triangle $\left(\binom{n}{k}\right)_{0\le n,k \le \infty}$, which begins
 $$\left(
\begin{array}{ccccccc}
 1 & 0 & 0 & 0 & 0 & 0 & 0 \\
 1 & 1 & 0 & 0 & 0 & 0 & 0 \\
 1 & 2 & 1 & 0 & 0 & 0 & 0 \\
 1 & 3 & 3 & 1 & 0 & 0 & 0 \\
 1 & 4 & 6 & 4 & 1 & 0 & 0 \\
 1 & 5 & 10 & 10 & 5 & 1 & 0 \\
 1 & 6 & 15 & 20 & 15 & 6 & 1 \\
\end{array}
\right),$$ is one of the most studied and used objects in mathematics. There are many ways to describe it. In this note, we use two ways. Firstly, we can regard it as the matrix representation of the Riordan array $\left(\frac{1}{1-x}, \frac{x}{1-x}\right)$. This means that we have
$$\binom{n}{k}=[x^n] \frac{1}{1-x} \left(\frac{x}{1-x}\right)^k.$$ Here, $[x^n]$ is the functional that operates on power series $f(x)$: $[x^n]$ extracts the coefficient $a_n$ of $x^n$ in the expansion of $f(x)$.

The second way is the following. We have
$$\binom{n}{k}=[x^{n-k}] (1+x)^n.$$ We call this a central representation of Pascal's triangle.
Other familiar and not so familiar number triangles can be represented in this way. For instance, the triangle with general $(n,k)$-th term given by
$$[x^{n-k}](1+x)^{2n}=[x^{n-k}] (1+2x+x^2)^n$$ is the triangle $\left(\binom{2n}{n+k}\right)$ that begins
$$\left(
\begin{array}{ccccccc}
 1 & 0 & 0 & 0 & 0 & 0 & 0 \\
 2 & 1 & 0 & 0 & 0 & 0 & 0 \\
 6 & 4 & 1 & 0 & 0 & 0 & 0 \\
 20 & 15 & 6 & 1 & 0 & 0 & 0 \\
 70 & 56 & 28 & 8 & 1 & 0 & 0 \\
 252 & 210 & 120 & 45 & 10 & 1 & 0 \\
 924 & 792 & 495 & 220 & 66 & 12 & 1 \\
\end{array}
\right).$$
This is the Riordan array $\left(\frac{1}{\sqrt{1-4x}}, x c(x)^2\right)$, where $c(x)=\frac{1-\sqrt{1-4x}}{2x}$ is the generating function of the Catalan numbers $C_n=\frac{1}{n+1} \binom{2n}{n}$.

The triangle $\left(\binom{2n}{n+k}\right)$ is the horizontal half \cite{halves} of Pascal's triangle.

In turn, Pascal's triangle is the horizontal half of the triangle whose $(n,k)$-th term is given by
$$[x^{n-k}] (1+x)^{n/2}.$$ This triangle begins
$$\left(
\begin{array}{ccccccc}
 \mathbf{1} & 0 & 0 & 0 & 0 & 0 & 0 \\
 \frac{1}{2} & 1 & 0 & 0 & 0 & 0 & 0 \\
 0 & \mathbf{1} & \mathbf{1} & 0 & 0 & 0 & 0 \\
 -\frac{1}{16} & \frac{3}{8} & \frac{3}{2} & 1 & 0 & 0 & 0 \\
 0 & 0 & \mathbf{1} & \mathbf{2} & \mathbf{1} & 0 & 0 \\
 \frac{3}{256} & -\frac{5}{128} & \frac{5}{16} & \frac{15}{8} & \frac{5}{2} & 1 & 0 \\
 0 & 0 & 0 & \mathbf{1} & \mathbf{3} & \mathbf{3} & \mathbf{1} \\
\end{array}
\right).$$
This is the Riordan array
$$\left(\frac{x+\sqrt{x^2+4}}{\sqrt{x^2+4}}, x\frac{x+\sqrt{x^2+4}}{\sqrt{x^2+4}}\right).$$
We see that the central description of these Riordan arrays is more compact than the traditional method.
In this note, we explore this central description in the general context.
\section{Theoretical  results}
In the following list of propositions, we outline the theory of the central description of Riordan arrays.
\begin{proposition} Let $(t_{n,k})_{0 \le n,k \le \infty}$ be the lower-triangular matrix defined by
$$t_{n,k}=[x^{n-k}] g(x)f(x)^n.$$
Then $(t_{n,k})$ is the Riordan array
$$\left(g\left(\Rev\left(\frac{x}{f(x)}\right)\right)\frac{x \left(\Rev\left(\frac{x}{f(x)}\right)\right)'}{\Rev\left(\frac{x}{f(x)}\right)},\Rev\left(\frac{x}{f(x)}\right)\right).$$
\end{proposition}
\begin{proof}
We have
\begin{align*}
[x^{n-k}]g(x)f(x)^n &=[x^n]x^k g(x)f(x)^n \\
&=[x^n] x^k g(x) \left(\frac{x}{\frac{x}{f(x)}}\right)^n\\
&=[x^{n-1}]x^{k-1} g(x)\left(\frac{x}{\frac{x}{f(x)}}\right)^n\\
&= n \cdot \frac{1}{n} [x^{n-1}] x^{k-1}g(x)\left(\frac{x}{\frac{x}{f(x)}}\right)^n\\
&= n [x^n] G\left(\Rev\left(\frac{x}{f(x)}\right)\right),\end{align*}
where $G(x)$ is such that
$$G'(x)=x^{k-1}g(x).$$
Then
\begin{align*}
[x^{n-k}]g(x)f(x)^n&=n [x^n] G(\Rev\left(\frac{x}{f(x)}\right)\\
&=[x^{n-1}]\frac{d}{dx} G\left(\Rev\left(\frac{x}{f(x)}\right)\right)\\
&=[x^{n-1}] G'\left(\Rev\left(\frac{x}{f(x)}\right)\right)\cdot \frac{d}{dx} \Rev\left(\frac{x}{f(x)}\right)\\
&=[x^{n-1}] g\left(\Rev\left(\frac{x}{f(x)}\right)\right)\left(\Rev\left(\frac{x}{f(x)}\right)\right)^{k-1} \left(\Rev\left(\frac{x}{f(x)}\right)\right)'\\
&=[x^n] x g\left(\Rev\left(\frac{x}{f(x)}\right)\right)\left(\Rev\left(\frac{x}{f(x)}\right)\right)^{k-1} \left(\Rev\left(\frac{x}{f(x)}\right)\right)'\\
&= [x^n]g\left(\Rev\left(\frac{f(x)}{x}\right)\right)\frac{x \left(\Rev\left(\frac{x}{f(x)}\right)\right)'}{\Rev\left(\frac{x}{f(x)}\right)}\left(\Rev\left(\frac{x}{f(x)}\right)\right)^k.\end{align*}
\end{proof}

\begin{corollary}  Let $(t_{n,k})_{0 \le n,k \le \infty}$ be the lower-triangular matrix defined by
$$t_{n,k}=[x^{n-k}] g(x)f(x)^n.$$
Then $(t_{n,k})$ is the Riordan array
$$\left(\frac{x \left(\Rev\left(\frac{x}{f(x)}\right)\right)'}{\Rev\left(\frac{x}{f(x)}\right)},\Rev\left(\frac{x}{f(x)}\right)\right)\cdot (g(x),x).$$
\end{corollary}

\begin{corollary}  Let $(t_{n,k})_{0 \le n,k \le \infty}$ be the lower-triangular matrix defined by
$$t_{n,k}=[x^{n-k}] g(x)f(x)^n.$$ Then the inverse matrix $(t_{n,k})^{-1}$ is given by the product
$$\left(\frac{1}{g(x)}, x\right)\cdot \left(\frac{x \left(\frac{x}{f(x)}\right)'}{\left(\frac{x}{f(x)}\right)}, \frac{x}{f(x)}\right)=\left(\frac{1}{g(x)}\frac{x \left(\frac{x}{f(x)}\right)'}{\left(\frac{x}{f(x)}\right)},\frac{x}{f(x)}\right).$$
\end{corollary}

\begin{corollary}
Consider the lower-triangular matrix defined by
$$\tilde{t}_{n,k}= [x^k] g(x)f(x)^n.$$
Then the matrix $\left(\tilde{t}_{n,k}\right)$ is the reversal of the Riordan array
$$\left(g\left(\Rev\left(\frac{x}{f(x)}\right)\right)\frac{x \left(\Rev\left(\frac{x}{f(x)}\right)\right)'}{\Rev\left(\frac{x}{f(x)}\right)},\Rev\left(\frac{x}{f(x)}\right)\right).$$
\end{corollary}

In the next proposition, we begin with a Riordan array $(u, v)$ and we find the elements $g, f$ of the central description of $(u,v)$.
\begin{proposition} Let $t_{n,k}$ be the $(n,k)$-element of the Riordan array $(u, v)$. Then
$$t_{n,k}=[x^{n-k}] g(x)f(x)^n,$$ where
$$f(x)=\frac{x}{\Rev(v(x))},\quad \text{and} \quad g(x)=f(x) \frac{u\left(\Rev(v(x))\right)}{v'\left(\Rev(v(x))\right)}.$$
\end{proposition}
We can also write this as
$$f(x)=\frac{x}{\Rev(v(x))},\quad \text{and} \quad g(x)=f(x) u\left(\Rev(v(x))\right)\left(\Rev(v(x))\right)'.$$
\begin{proof} We have
$$v(x)=\Rev\left(\frac{x}{f(x)}\right),$$ and hence
$$f(x)=\frac{x}{\Rev(v(x))}.$$
We have
$$u(x)=g(v)x \frac{v'}{v}.$$
Solving for $g(v)$ gives
$$g(v(x))=u(x) \frac{v(x)}{x v'(x)}.$$
Now setting $x=\Rev{v}$ allows us to solve for $g(x)$.
\end{proof}
\begin{example} We consider the Riordan array $(u, v)=\left(\frac{1}{1-x-x^2}, \frac{x}{1-2x}\right)$.
We have
$$v(x)=\frac{x}{1-2x} \Longrightarrow \Rev(v)(x)=\frac{x}{1+2x}.$$ In addition, we have
$$v'(x)=\frac{1}{(1-2x)^2}.$$
We find that
$$f(x)=\frac{x}{\Rev(v)(x)}=1+2x.$$
Also
$$g(x)=u\left(\frac{x}{1+2x}\right) \frac{x}{\frac{x}{1+2x} \cdot \frac{1}{(1-2\frac{x}{1+2x})^2}}$$ and hence we have $$g(x)=\frac{1+2x}{1+3x+x^2}.$$
Thus we have
$$[x^n]\frac{1}{1-x-x^2}\left(\frac{x}{1-2x}\right)^k=[x^{n-k}]\frac{1+2x}{1+3x+x^2}(1+2x)^n.$$
\end{example}
\begin{proposition} The Riordan array with general $(n,k)$-element $t_{n,k}$ given by
 $$t_{n,k}= [x^{n-k}] g(x)f(x)^n$$ is the vertical half of the Riordan array $(g(x), xf(x))$.
\end{proposition}
\begin{proof}
We have
\begin{align*}
[x^{n-k}] g(x)f(x)^n &= [x^{n-k}] \frac{1}{x^n} g(x)(xf(x))^n\\
&=[x^{2n-k}] g(x)(xf(x))^n.\end{align*}
\end{proof}

\section{Relation to the $A$-sequence and the $Z$-sequence}
To each Riordan array $(u, v)$  is associated an $A$-sequence with generating function $A(x)$ and a $Z$-sequence with generating function $Z(x)$. We have
$$A(x)=\frac{x}{\Rev(v)} \quad \text{and}\quad Z(x)=\frac{1}{\Rev(v)}\left(1-\frac{1}{u(\Rev(v))}\right).$$
We then have the following results.
\begin{proposition} Given a Riordan array $\left(t_{n,k}\right)$ defined by
$$t_{n,k}=[x^{n-k}] g(x)f(x)^n,$$ we have
$$A(x)=f(x) \quad \text{and} \quad Z(x)=\frac{f(x)}{x}\left(1-\frac{f(x)\frac{d}{dx}\left(\frac{x}{f(x)}\right)}{g(x)}\right).$$
\end{proposition}
We can also express $Z(x)$ as
$$Z(x)=\frac{f}{x}\left(1-\frac{1}{g}\left(1-\frac{xf'}{f}\right)\right).$$
\begin{proposition}
Given a Riordan array $(u,v)$ with $A$-sequence $A(x)$ and $Z$-sequence $Z(x)$, we have
$$f(x)=A(x) \quad\text{and}\quad g(x)=\frac{A-xA'}{A-xZ}.$$
\end{proposition}
Thus we have
$$t_{n,k}=[x^{n-k}] \frac{A-xA'}{A-xZ} A(x)^n.$$
Equivalently, we can say that the Riordan array defined by the $A$-sequence $A(x)$ and the $Z$-sequence $Z(x)$
is the reversal of the lower-triangular matrix defined by
$$\tilde{t}_{n,k}=[x^n] \frac{A-xA'}{A-xZ} A(x)^n.$$
We note that when $A'=Z$, we have
$$t_{n,k}=[x^{n-k}]A(x)^n.$$ This occurs only within the hitting-time subgroup of the Riordan group. Thus we have the following result.
\begin{proposition} A Riordan array $(u, v)$ with general $(n,k)$-th term $t_{n,k}=[x^n]uv^k$ is an element of the hitting time subgroup if and only if
$$t_{n,k}=[x^{n-k}]A(x)^n,$$ where $A(x)$ is the generating function of the $A$-sequence of $(u,v)$. This is equivalent to asking that
$$t_{n,k}=[x^{n-k}] \left(\frac{x}{\Rev(v)(x)}\right)^n.$$
\end{proposition}
\section{A vertical antecedent}
We have the following result.
\begin{proposition} The inverse of the number triangle $\left(t_{n,k}\right)$ defined by
$$t_{n,k}=[x^{n-k}] g(x)f(x)^n$$ is the vertical half of the Riordan array given by
$$\left(\frac{1}{g\left(\Rev\left(\frac{x}{f}\right)\right)}, \frac{x^2}{\Rev \left(\frac{x}{f}\right)}\right).$$ \end{proposition}
\begin{example} We consider Pascal's triangle with $t_{n,k}=\binom{n}{k}$. Then $f(x)=1+x$ and $g(x)=1$.
We have $$\frac{x}{f(x)}=\frac{x}{1+x} \Longrightarrow \Rev \left(\frac{x}{f}\right)=\frac{x}{1-x},$$ and hence
$$\frac{x^2}{\Rev \left(\frac{x}{f}\right)}=x(1-x).$$ We also have
$$\frac{1}{g\left(\Rev\left(\frac{x}{f}\right)\right)}=1.$$
Thus we obtain the Riordan array $(1, x(1-x))$ whose vertical half is the inverse Pascal matrix with general term
$(-1)^{n-k}\binom{n}{k}$. The Riordan array $(1, x(1-x))$ begins
$$\left(
\begin{array}{cccccccc}
 1 & 0 & 0 & 0 & 0 & 0 & 0 & 0 \\
 0 & 1 & 0 & 0 & 0 & 0 & 0 & 0 \\
 0 & -1 & 1 & 0 & 0 & 0 & 0 & 0 \\
 0 & 0 & -2 & 1 & 0 & 0 & 0 & 0 \\
 0 & 0 & 1 & -3 & 1 & 0 & 0 & 0 \\
 0 & 0 & 0 & 3 & -4 & 1 & 0 & 0 \\
 0 & 0 & 0 & -1 & 6 & -5 & 1 & 0 \\
 0 & 0 & 0 & 0 & -4 & 10 & -6 & 1 \\
\end{array}
\right).$$
\end{example}
\begin{example} We take the example of $g(x)=\sqrt{1-4x}, f(x)=\frac{1}{c(x)}$.
Then the Riordan array defined by $g(x)$ and $f(x)$ with general term
$$t_{n,k}=[x^{n-k}] \sqrt{1-4x} \left(\frac{1}{c(x)}\right)^n$$ begins
$$\left(
\begin{array}{ccccccc}
 1 & 0 & 0 & 0 & 0 & 0 & 0 \\
 -3 & 1 & 0 & 0 & 0 & 0 & 0 \\
 1 & -4 & 1 & 0 & 0 & 0 & 0 \\
 1 & 4 & -5 & 1 & 0 & 0 & 0 \\
 1 & 0 & 8 & -6 & 1 & 0 & 0 \\
 1 & 0 & -4 & 13 & -7 & 1 & 0 \\
 1 & 0 & 0 & -12 & 19 & -8 & 1 \\
\end{array}
\right),$$ with inverse which begins
$$\left(
\begin{array}{ccccccc}
 1 & 0 & 0 & 0 & 0 & 0 & 0 \\
 3 & 1 & 0 & 0 & 0 & 0 & 0 \\
 11 & 4 & 1 & 0 & 0 & 0 & 0 \\
 42 & 16 & 5 & 1 & 0 & 0 & 0 \\
 163 & 64 & 22 & 6 & 1 & 0 & 0 \\
 638 & 256 & 93 & 29 & 7 & 1 & 0 \\
 2510 & 1024 & 386 & 130 & 37 & 8 & 1 \\
\end{array}
\right).$$
Now we have
$$\frac{x}{f(x)}=xc(x) \Longrightarrow \Rev\left(\frac{x}{f(x)}\right)=x(1-x),$$ and
hence $$\frac{x^2}{\frac{x}{f(x)}}=\frac{x}{1-x}.$$
We also have $$\frac{1}{g\left(\Rev\left(\frac{x}{f(x)}\right)\right)}=\frac{1}{1-2x}.$$
The Riordan array $\left(\frac{1}{1-2x}, \frac{x}{1-x}\right)$ then begins
$$\left(
\begin{array}{ccccccc}
 \mathbf{1} & 0 & 0 & 0 & 0 & 0 & 0 \\
 2 & \mathbf{1} & 0 & 0 & 0 & 0 & 0 \\
 4 & \mathbf{3} & \mathbf{1} & 0 & 0 & 0 & 0 \\
 8 & 7 & \mathbf{4} & \mathbf{1} & 0 & 0 & 0 \\
 16 & 15 & \mathbf{11} & \mathbf{5} & \mathbf{1} & 0 & 0 \\
 32 & 31 & 26 & \mathbf{16} & \mathbf{6} & \mathbf{1} & 0 \\
 64 & 63 & 57 & \mathbf{42} & \mathbf{22} & \mathbf{7} & \mathbf{1} \\
\end{array}
\right).$$
\end{example}

\section{A structure for the central description}
In this section, we wish to explore the group structure of Riordan arrays in terms of the central description.

We let $\{g, f\}$ denote the Riordan array $\left(t_{n,k}\right)$ whose elements are given by the central description
$$t_{n,k}=[x^{n-k}] g(x)f(x)^n.$$ It is immediate that
$$\{1, 1\}=(1,x)$$ is the identity matrix. Furthermore, we have
$$\{g(x), 1\}=(g(x),x),$$   and
$$\{1, f(x)\}=\left(\frac{x \left(\Rev\left(\frac{x}{f}\right)\right)'}{\Rev\left(\frac{x}{f}\right)}, \Rev\left(\frac{x}{f}\right)\right)=\left(\frac{x \left(\frac{x}{f}\right)'}{\frac{x}{f}}, \frac{x}{f}\right)^{-1}.$$

We then ask the question: given two Riordan arrays $\{g_1, f_1\}$ and $\{g_2, f_2\}$, what are $g_3, f_3$ in the product
$$\{g_1, f_1\} \cdot \{g_2, f_2\} = \{g_3, f_3\}?$$
An initial response is the following.
\begin{proposition} We have
$$\{1, f_1\} \cdot \{1, f_2\} = \{1, \frac{x}{\left(\frac{x}{f_1(x)}\right) \circ \left(\frac{x}{f_2(x)}\right)(x)}\}.$$
\end{proposition}
\begin{example}
We have
\begin{align*}\{1, \frac{1}{1-x}\} \cdot \{1, 1+2x\}&=\{1, \frac{x}{(x(1-x))\circ (\frac{x}{1+2x})(x)}\}\\
&=\{1, \frac{x}{\frac{x}{1+2x}(1-\frac{x}{1+2x})}\}\\
&=\{1, \frac{x}{\frac{x(1+x)}{(1+2x)^2}}\} \\
&=\{1, \frac{(1+2x)^2}{1+x}\}.\end{align*}
\end{example}
We can deduce from the above product law that
$$\{1, f(x)\}^{-1}=\{1, \frac{x}{\Rev\left(\frac{x}{f(x)}\right)}\}.$$
\begin{proposition} We have
$$\{g, 1\}\cdot \{1, f\}=\{g\left(\frac{x}{f(x)}\right), f(x)\}.$$
\end{proposition}
\begin{proof}
We have
\begin{align*}
\{g, 1\}\cdot \{1, f\}&=(g(x),x) \cdot \left(\frac{x \left(\Rev\left(\frac{x}{f(x)}\right)\right)'}{\Rev\left(\frac{x}{f}\right)}, \Rev\left(\frac{x}{f}\right)\right)\\
&=\left(g(x) \frac{x \left(\Rev\left(\frac{x}{f(x)}\right)\right)'}{\Rev\left(\frac{x}{f}\right)}, \Rev\left(\frac{x}{f}\right)\right)\\
&=\left(G\left(\Rev\left(\frac{x}{f}\right)\right)\frac{x \left(\Rev\left(\frac{x}{f(x)}\right)\right)'}{\Rev\left(\frac{x}{f}\right)}, \Rev\left(\frac{x}{f}\right)\right)\\
&=\{G(x), f(x)\}, \end{align*}
where $$g(x)=G\left(\Rev\left(\frac{x}{f}\right)\right) \Longrightarrow G(x)=g\left(\frac{x}{f(x)}\right).$$
\end{proof}
\begin{corollary} We have
$$ \{g(x), f(x)\}=\{g\left(\Rev\left(\frac{x}{f(x)}\right)\right), 1\}\cdot \{1, f(x)\}.$$
\end{corollary}
The general product formula is as follows.
\begin{proposition}
We have
$$\{g_1, f_1\}\cdot \{g_2, f_2\}= \{g_1\left(\frac{x}{f_2(x)}\right)g_2(x), \frac{x}{\left(\frac{x}{f_1(x)}\right) \circ \left(\frac{x}{f_2(x)}\right)}\}.$$
\end{proposition}
\begin{corollary} We have
$$\{g, f\}^{-1}=\{\frac{1}{g\left(\Rev\left(\frac{x}{f}\right)\right)}, \frac{x}{\Rev\left(\frac{x}{f}\right)}\}.$$
\end{corollary}
\begin{example}
We have
$$\{1+x+x^2, \frac{1}{1-x}\}^{-1}=\{\frac{1}{2-x-\sqrt{1-4x}}, \frac{1}{c(x)}\}.$$
\end{example}
\begin{example} We take the example of the array $(u, v)=\left(\frac{1}{1+x+x^2}, \frac{x}{1+x+x^2}\right)$. This is the coefficient array of a family of orthogonal polynomials \cite{classical}, whose moments are the Motzkin numbers. In fact, we have
$$\left(\frac{1}{1+x+x^2}, \frac{x}{1+x+x^2}\right)^{-1}=\left(\frac{1-x-\sqrt{1-2x-3x^2}}{2x^2}, \frac{1-x-\sqrt{1-2x-3x^2}}{2x}\right).$$
Then we have
$$(u, v)=\{\frac{1-2x-3x^2+(1-x)\sqrt{1-2x-3x^2}}{2(1-2x-3x^2)}, \frac{1-x+\sqrt{1-2x-3x^2}}{2}\}.$$
The corresponding moment matrix $(u, v)^{-1}$ is then given by
$$(u, v)^{-1}= \{1-x^2, 1+x+x^2\}.$$
Thus in the central representation, the moment matrix has a simple description.

We find in particular that the Motzkin numbers $M_n$ can be described as follows:
$$M_n = [x^n] (1-x^2)(1+x+x^2)^n.$$
More generally, we have the following proposition.
\begin{proposition} The moments $\mu_n$ of the family of generalized Chebyshev polynomials whose coefficient matrix is given by the Riordan array
$$\left(\frac{1- s x - t x^2}{1+ ax + b x^2}, \frac{x}{1+a x + b x^2}\right)$$ satisfy
$$ \mu_n =[x^n] \frac{1-bx^2}{1- sx-tx^2} (1+ax+bx^2)^n.$$
\end{proposition}
\begin{proof} The inverse matrix $(u, v)=\left(\frac{1- s x - t x^2}{1+ ax + b x^2}, \frac{x}{1+a x + b x^2}\right)^{-1}$ has a central description given by
$$g(x)=\frac{1-bx^2}{1- sx-tx^2},\quad f(x)=1+ax+bx^2.$$
\end{proof}
For instance, we have the following
\begin{align*} M_n &=[x^n](1-x^2)(1+x+x^2)^n\\
C_{n+1}&=[x^n] (1-x^2)(1+2x+x^2)^n\\
\binom{n}{\lfloor \frac{n}{2} \rfloor}&=[x^n](1+x)(1+x^2)^n\\
\binom{2n}{n}&=[x^n](1-2x^2)(1+2x+2x^2)^n.\end{align*}
\end{example}

\section{Some common Riordan arrays}
The following table gives a list of Riordan arrays $(u, v)$ and their central description $\{g, f\}$.
\begin{center}
\begin{tabular}{|c|c|} \hline
$(u, v)$ & $\{g, f\}$ \\ \hline
$\left(\frac{1}{1-rx}, \frac{x}{1-rx}\right)$ & $\{1, 1+rx\}$ \\ \hline
$\left(1, \frac{x}{1-x}\right)$ & $\{\frac{1}{1+x}, 1+x\}$ \\ \hline
$(1, x(1-x))$ & $\{\frac{1}{c(x)\sqrt{1-4x}}, \frac{1}{c(x)}\}$ \\ \hline
$(1-x, x(1-x))$ & $\{\frac{1}{c(x)^2\sqrt{1-4x}}, \frac{1}{c(x)}\}$ \\ \hline
$\left(\frac{1}{1-x}, \frac{x(1+x)}{1-x}\right)$ & $\{\frac{1+x+\sqrt{1+6x+x^2}}{2\sqrt{1+6x+x^2}}, \frac{1+x+\sqrt{1+6x+x^2}}{2}\}$\\ \hline
$\left(\frac{1}{1-x}, \frac{x(1+x)}{1-x}\right)^{-1}$ & $\{\frac{1+2x-x^2}{1+x}, \frac{1-x}{1+x}\}$\\ \hline
$\left(\frac{1}{\sqrt{1-4x}}, xc(x)\right)$ & $\{\frac{1}{1-x}, \frac{1}{1-x}\}$ \\ \hline
$\left(\frac{1}{\sqrt{1-4rx}}, xc(rx)\right)$ & $\{\frac{1}{1-rx}, \frac{1}{1-rx}\}$ \\ \hline
$ (1, xc(x))$ & $\{\frac{1-2x}{1-x}, \frac{1}{1-x}\}$ \\ \hline
$ (c(x), x c(x))$ & $\{\frac{1-2x}{(1-x)^2},\frac{1}{1-x}\}$ \\ \hline
\end{tabular}
\end{center}
Note that this table can also be read in the following way: the left hand element $(u,v)$ is the vertical half of the Riordan array $(g(x), xf(x))$.

For instance, $(1-x, x(1-x))$, or equivalently, $\{\frac{1}{c(x)^2\sqrt{1-4x}}, \frac{1}{c(x)}\}$, is the vertical half of the Riordan array $\left(\frac{1}{c(x)^2\sqrt{1-4x}}, \frac{x}{c(x)}\right)$, which begins
$$\left(
\begin{array}{ccccccccc}
 \mathbf{1} & 0 & 0 & 0 & 0 & 0 & 0 & 0 & 0 \\
 0 & \mathbf{1} & 0 & 0 & 0 & 0 & 0 & 0 & 0 \\
 1 & \mathbf{-1} & \mathbf{1} & 0 & 0 & 0 & 0 & 0 & 0 \\
 4 & 0 & \mathbf{-2} & \mathbf{1} & 0 & 0 & 0 & 0 & 0 \\
 15 & 1 & \mathbf{0} & \mathbf{-3} & \mathbf{1} & 0 & 0 & 0 & 0 \\
 56 & 5 & 0 & \mathbf{1} & \mathbf{-4} & \mathbf{1} & 0 & 0 & 0 \\
 210 & 21 & 1 & \mathbf{0} & \mathbf{3} & \mathbf{-5} & \mathbf{1} & 0 & 0 \\
 792 & 84 & 6 & 0 & \mathbf{0} & \mathbf{6} & \mathbf{-6} & \mathbf{1} & 0 \\
 3003 & 330 & 28 & 1 & \mathbf{0} & \mathbf{-1} & \mathbf{10} & \mathbf{-7} & \mathbf{1} \\
\end{array}
\right).$$
The vertical half of this matrix begins
$$\left(
\begin{array}{cccccc}
 1 & 0 & 0 & 0 & 0 & 0 \\
 -1 & 1 & 0 & 0 & 0 & 0 \\
 0 & -2 & 1 & 0 & 0 & 0 \\
 0 & 1 & -3 & 1 & 0 & 0 \\
 0 & 0 & 3 & -4 & 1 & 0 \\
 0 & 0 & -1 & 6 & -5 & 1 \\
\end{array}
\right).$$
This is $(1-x, x(1-x))$.

Similarly, the Riordan array $\left(\frac{1-2x}{1-x}, \frac{x}{1-x}\right)$ which begins
$$\left(
\begin{array}{ccccccccc}
 1 & 0 & 0 & 0 & 0 & 0 & 0 & 0 & 0 \\
 -1 & 1 & 0 & 0 & 0 & 0 & 0 & 0 & 0 \\
 -1 & 0 & 1 & 0 & 0 & 0 & 0 & 0 & 0 \\
 -1 & -1 & 1 & 1 & 0 & 0 & 0 & 0 & 0 \\
 -1 & -2 & 0 & 2 & 1 & 0 & 0 & 0 & 0 \\
 -1 & -3 & -2 & 2 & 3 & 1 & 0 & 0 & 0 \\
 -1 & -4 & -5 & 0 & 5 & 4 & 1 & 0 & 0 \\
 -1 & -5 & -9 & -5 & 5 & 9 & 5 & 1 & 0 \\
 -1 & -6 & -14 & -14 & 0 & 14 & 14 & 6 & 1 \\
\end{array}
\right),$$ has a vertical half which begins
$$\left(
\begin{array}{cccccc}
 1 & 0 & 0 & 0 & 0 & 0 \\
 0 & 1 & 0 & 0 & 0 & 0 \\
 0 & 1 & 1 & 0 & 0 & 0 \\
 0 & 2 & 2 & 1 & 0 & 0 \\
 0 & 5 & 5 & 3 & 1 & 0 \\
 0 & 14 & 14 & 9 & 4 & 1 \\
\end{array}
\right).$$
This is $(1, xc(x))$.

\section{The exponential case}
We recall that an exponential Riordan array $[u, v]$ is defined by two exponential generating functions 
$$u(x)=\sum_{n=0}^{\infty} u_n \frac{x^n}{n!},$$ where $u_0 \ne 0$, and 
$$v(x)=\sum_{n=1}^{\infty} v_n \frac{x^n}{n!},$$ where $v_1 \ne 0$. 
Then the matrix with general $(n,k)$-term 
$$t_n = \frac{n!}{k!} [x^n] u(x)v(x)^k$$ is the invertible lower-triangular matrix that represents the exponential Riordan array $[u, v]$. 
In similar fashion, we can define, for 
$$g(x)=\sum_{n=0}^{\infty} g_n \frac{x^n}{n!},$$ where $g_0 \ne 0$, and 
$$f(x)=\sum_{n=0}^{\infty} f_n \frac{x^n}{n!},$$ where $f_0 \ne 0$, an invertible lower-triangular matrix whose $(n,k)$-th term is given by 
$$t_{n,k}=\frac{n!}{k!} [x^{n-k}] g(x)f(x)^n.$$ 
We shall denote this matrix by $\{g, f\}_e$. 

It is then possible to carry out a similar analysis as above. 
\begin{example} 
We have the following equalities 
$$\{1+rx, e^x\}_e = \left[\frac{1-r W(-x)}{1-W(-x)}, -W(-x)\right]=\left[\frac{1-x}{1+rx}, xe^{-x}\right]^{-1}.$$
Here, $W$ is the principal determination of the Lambert $W$ function. 
\end{example}

\section{Conclusions} The use of the central description of a Riordan array can often give an alternative and insightful perspective on the properties of that array. From a computational point of view, we can use whichever description provides the easiest calculation. For instance, the elements of the well known Catalan matrix $(c(x), xc(x))$ can be described by
$$[x^n] c(x)(x c(x))^k = [x^{n-k}]  c(x)^{k+1},$$ or alternatively
$$[x^{n-k}] \frac{1-2x}{(1-x)^2} \left(\frac{1}{1-x}\right)^n.$$
The relatively complicated expression $c(x)=\frac{1-\sqrt{1-4x}}{2x}$ with an expensive square root function makes the second format easier to compute. The central expression may also serve to motivate novel directions of exploration. For instance, we may wish to look at the family of arrays described by
$$[x^{n-k}] \frac{1-(r+1)x}{(1-rx)^2} \left(\frac{1}{1-rx}\right)^n.$$
The Riordan array with general term given by
$$[x^{n-k}] \frac{1-3x}{(1-2x)^2} \left(\frac{1}{1-2x}\right)^n$$ turns out to be the Riordan array
$$\left(\frac{3-c(2x)}{2(1-4x c(2x))}, xc(2x)\right).$$
This array has its inverse given by
$$\left(\frac{(1-2x)(1-4x)}{1-3x}, x(1-2x)\right).$$
The simple variant
$$[x^{n-k}] \frac{1-2x}{1-x} \left(\frac{1}{1-2x}\right)^n$$ describes the Riordan array whose $(u,v)$ representation is given by
$$\left(\frac{1-6x-16x^2+(1-2x)\sqrt{1-8x}}{2(1+x)(1-8x)}, xc(2x)\right).$$
The even simpler variant
$$[x^{n-k}] \frac{1-2x}{1-x} \left(\frac{1}{1-x}\right)^n$$ gives the variant $(1, xc(x))$ of the Catalan triangle $(c(x), xc(x))$.

\bigskip
\hrule
\bigskip
\noindent 2010 {\it Mathematics Subject Classification}: Primary
15A30; Secondary 15B36, 11C20, 20G05, 20H25.
\noindent \emph{Keywords:} Riordan array, Riordan group, central coefficients, generating function, integer sequence

\end{document}